\documentclass[11pt,letterpaper]{amsart}

\usepackage{amssymb, longtable,comment}
\usepackage[mathcal]{euscript}
\usepackage{color,cite}

\usepackage{hyperref}                   
\hypersetup{colorlinks,
    linkcolor=blue,%
    citecolor=blue}

\newcommand{\cA}{\mathcal{A}}
\newcommand{\cB}{\mathcal{B}}

\newcommand{\cM}{\mathcal{M}}

\newcommand{\MM}{\mathbb{M}}

\numberwithin{equation}{section}

\newcommand{\cX}{\mathcal{X}}

\newtheorem{theorem}{Theorem}[section]
\newtheorem{proposition}[theorem]{Proposition}
\newtheorem*{proposition*}{Proposition}

\theoremstyle{definition}

\theoremstyle{remark} \newtheorem{remark}[theorem]{Remark}

\begin{document}

\title[Characterizing convex trace ranges]{Characterizing convex trace ranges  in finite atomic von Neumann algebras}

\author[A. Arziev]{Allabay Arziev}
\address{V.I.Romanovskiy Institute of Mathematics,     Uzbekistan Academy of Sciences}
\email{allabayarziev@gmail.com}

\author[K. Kudaybergenov]{Karimbergen Kudaybergenov}
\address{North-Caucasus Center for Mathematical Research
of the Vladikavkaz Scientific Centre of the Russian Academy of
Sciences}
\email{kudaybergenovkk@gmail.com}

\begin{abstract} The paper is devoted to characterizing convex trace ranges  in finite atomic von Neumann algebras.
The main result provides us with the necessary and sufficient condition for the range of a faithful normal trace on a finite atomic von Neumann algebra to be convex.
In order to prove this result we will prove the following result, which has independent interest.
Let ${\bf a}=(a_1, \ldots, a_n, \ldots)$ be a non-increasing positive sequence such that $\sum\limits_{n=1}^\infty a_n=1.$ Then each real number $0\le r \le 1$ can be represented in the form
\(
r=\sum\limits_{n=1}^\infty \varepsilon_n a_n, \,\,\, \varepsilon_n \in \{0,1\}, n\ge 1,
\)
if and only if the sequence ${\bf a}$ satisfies \(a_n \le 1-\sum\limits_{k=1}^n a_k
\)
for all $n\ge 1.$ A set $K$ of all sequences that satisfy the last property can be represented as a convex weak-compact subset of $\ell_1 = c_0^*$. We will describe the set of all extreme points of $K.$

\end{abstract}

\maketitle

\bigskip

\section{Introduction}
Recall that
 for a
   Banach space  $\cX$ and   a measure space $(\Omega, \Sigma)$, where $\Sigma$ is a $\sigma$-algebra of subsets of $\Omega$,
the classical Lyapunov convexity theorem \cite[Theorem I]{Lia40} states that if $\cX$ is finite-dimensional and if $\mu$ is a non-atomic $\sigma$-additive $\cX$-valued measure defined on $\Sigma$ then its range $\mu(\Sigma)$
is compact and convex.
Along with the Brouwer--Kakutani--Fan--Glicksberg fixed point theorems,
and several versions of the Hahn--Banach theorem, Lyapunov's theorem~\cite{Lia40, Lia46} on the range
of a nonatomic finite-dimensional vector measure
plays a crucial role in  modern
mathematical economics, especially in  general equilibrium and game theory (see \cite{AA91, DU77, Holmes, KK75,LT2}).

Let $\cM$ be an atomic von Neumann algebra with a faithful normal tracial state $\tau.$
We can find a sequence of mutually orthogonal minimal projections
$e_1, \ldots, e_n, \ldots$ in $\cM$ such that
$\sum\limits_{n=1}^\infty e_n=\mathbf{1}.$

Let
\begin{align}\label{tn}
t_n:=\tau(e_n),\, n \ge 1.
\end{align} If necessary, after enumerating, we may assume that
 $t_1\ge t_2\ge \ldots t_n \ge \ldots.$

Let ${\bf a}=(a_1, \ldots, a_n, \ldots)$ be a non-increasing sequence of positive  real numbers such that
\begin{align}\label{sumt_n-}
a_n & \le \sum\limits_{k=n+1}^\infty a_k
\end{align}
for all $n\ge 1.$ If the series $\sum\limits_{n=1}^\infty a_n$ converges, then we can consider the sequence ${\bf a}$ as an element
of $\ell_1,$ and set
\begin{align*}
\|{\bf a}\|_1=\sum\limits_{n=1}^\infty a_n.
\end{align*}
In particular, if $\|{\bf a}\|_1=1,$ then  relations \eqref{sumt_n-} rewrites as follows
\begin{align}\label{sumt_n}
a_n & \le 1-\sum\limits_{k=1}^n a_k
\end{align}
for all $n\ge 1.$

The next result provides us with the necessary and sufficient condition for the range of a faithful normal trace on a finite atomic von Neumann algebra to be convex.

\begin{theorem}\label{atomLya} Let  $\cM$ be an atomic von Neumann algebra with a faithful normal tracial state $\tau$
and let $\{t_n\}_{n\ge 1}$ be a sequence defined as in \eqref{tn}. Then the following assertions are equivalent.
\begin{enumerate}
\item $\tau(P(\cM))$ is convex.
\item the sequence $\{t_n\}_{n\ge 1}$ satisfies~\eqref{sumt_n}.
\end{enumerate}
\end{theorem}

\section{Proof of the main result}

Let  $\cM$ be an atomic von Neumann algebra with a faithful normal tracial state $\tau.$ Then  it has the form
\begin{align*}
\cM & \equiv \bigoplus_{k\in F}\MM_k,
\end{align*}
where $\MM_k$ is the factor of type I$_k$ and $F$ is an infinite subset of the integers $\mathbb{N}.$

Let $\cA$ be a maximal abelian von Neumann subalgebra (masa) in $\cM.$
By \cite[Theorem 2.4.3 and Remark 2.4.4]{SS08}, for any masa $\cA$ in $\cM$ there is a unitary $w\in \cM$ such that $\cA=w\cB w^\ast.$
Since any projection $e$ in $\cM$ is contained in some masa, there exists a projection $q\in \cA$ such that
 $e\sim q,$ that is, $e=vqv^\ast$ for a unitary $v\in \cM.$  Thus,
 \begin{align*}
\tau(e) & =\tau(vqv^\ast)=\tau(q).
 \end{align*}
 Hence,
 \begin{align}\label{ED}
\tau\left(P(\cA)\right)=\tau\left(P(\cM)\right).
 \end{align}
Since $\cA$ is an atomic masa, we can find a sequence of mutually orthogonal minimal projections
$q_1, \ldots, q_n, \ldots$ in $\cA$ such that
$\sum\limits_{n=1}^\infty q_n=\mathbf{1}.$ Moreover, any projection $q\in \cA$ can be represented in the form of
\begin{align}\label{qn}
q=\sum\limits_{n=1}^\infty \varepsilon_n q_n,
\end{align}
 where $\varepsilon_n \in \{0,1\}$ for all $n\ge 1.$ Furthermore, instead the sequence $\{e_n\}_{n\ge 1}$ we can consider the sequence
 $\{q_n\}_{n\ge1},$ that is, we can set
\begin{align*}
t_n:=\tau(q_n),\, n \ge 1.
\end{align*}
If necessary, after enumerating, we may assume that
 $t_1\ge t_2\ge \ldots t_n \ge \ldots.$
Combing \eqref{ED} and \eqref{qn} we have that
\begin{align}\label{equal}
\tau(P(\cM)) & =\left\{\sum\limits_{n=1}^\infty \varepsilon_n t_n: \varepsilon_n \{0,1\}, n\ge 1\right\}.
\end{align}

Taking into account equation \eqref{equal}, we see that Theorem \ref{atomLya} follows directly from the following result, which has independent interest.

\begin{theorem}\label{ret} Let ${\bf a}=(a_1, \ldots, a_n, \ldots)$ be a positive sequence such that $\|{\bf a}\|_1=1.$ Then each real number $0\le r \le 1$ can be represented in the form
\begin{align}\label{dec-r}
r=\sum\limits_{n=1}^\infty \varepsilon_n a_n, \,\,\, \varepsilon_n \in \{0,1\}, n\ge 1,
\end{align}
if and only if the sequence ${\bf a}$ satisfies \eqref{sumt_n}.
\end{theorem}

\begin{proof} 'if' part. Let $r \in [0,1].$
Define  sequences
$\{\varepsilon_0, \ldots, \varepsilon_n,\ldots\}\subset \{0, 1\}$ and
$\{r_0, \ldots, r_n, \ldots\}\subset [0,1]$ by the following recurrence formula:
\begin{align*}
\varepsilon_0 & = 0,\,\,\, r_0 =0,\\
\varepsilon_n & =\begin{cases}
			0, & \text{if    } r-r_{n-1}<a_n\\
            1, & \text{otherwise}
		 \end{cases}, & n \ge 1,\\
r_n & = r_{n-1}+\varepsilon_n a_n, & \,\,  n\ge 1.
\end{align*}
Let us show that
\begin{align*}
r=\sum\limits_{n=1}^\infty \varepsilon_n a_n.
\end{align*}
It suffices to show that
\begin{align}\label{normxn}
0 \le &\,  r-r_n  \le 1-\sum\limits_{k=1}^n a_k,\, n\ge 1.
\end{align}
Let $n=1.$ If $a_1>r-r_0=r,$ we have that $\varepsilon_1=0.$ Hence, $r_1=r_0+\varepsilon_1a_1=0.$ Thus
\begin{align*}
0\le r-r_1 & = r <  a_1=1-\sum\limits_{k=2}^\infty a_k\stackrel{\eqref{sumt_n}}{\le} 1-a_1.
\end{align*}
If $a_1 \le r-r_0,$ we have that $\varepsilon_1=1$ and $r_1=r_0+a_1=a_1.$ Further, we obtain that
\begin{eqnarray*}
0\le r-r_1 & = & r-a_1\le 1-a_1,
\end{eqnarray*}
because $0\le r\le 1.$

Now,
 assume that \eqref{normxn} holds for $n-1,$
that is,
\begin{align}\label{indn-1}
0\le & \, r-r_{n-1} \le 1-\sum\limits_{k=1}^{n-1}a_k.
\end{align}

If $a_n>r-r_{n-1},$ we have that $\varepsilon_n=0,$ and therefore $r_n=r_{n-1}+\varepsilon_n a_n=r_{n-1}.$ Therefore
\begin{align*}
0\le r-r_n  & = r-r_{n-1} <  a_n\stackrel{\eqref{sumt_n}}{\le} 1-\sum\limits_{k=1}^{n}a_k.
\end{align*}
If $a_n\le r-r_{n-1},$ we have that $\varepsilon_n=1,$ and $r_n=r_{n-1}+a_n.$ Hence
\begin{eqnarray*}
0 & \le & r-r_n =r-(r_{n-1}+a_n)=r-r_{n-1}-a_n \stackrel{\eqref{indn-1}}{\le}1-\sum\limits_{k=1}^{n-1} a_k-a_n=1-\sum\limits_{k=1}^n a_k.
\end{eqnarray*}
Since $\sum\limits_{k=1}^\infty a_k=1,$ we have $r_n \uparrow r,$
and the proof of the 'if' is complete.

'only if' part.
 Assume by contrary, that is, there exists $n$ such that
 \[
\sum\limits_{k=n+1}^\infty a_k< a_n.
 \]
 Let us take a number $r$ such that
 \[
 \sum\limits_{k=n+1}^\infty a_k<r<a_n.
 \]
By the assumption we find expansion
\[
r: =\sum\limits_{k=1}^\infty \varepsilon_k a_k.
\]
If $\varepsilon_k=1$ for some $k\in \{1, \ldots, n\},$ we have that
\begin{align*}
r & \ge a_k\ge a_n>r.
\end{align*}
If $\varepsilon_k=0$ for all  $k\in \{1, \ldots, n\},$ we have that
\begin{align*}
r & =\sum\limits_{k=n+1}^\infty \varepsilon_k a_k
\le \sum\limits_{k=n+1}^\infty a_k<r.
\end{align*}
So, in both cases we have contradiction. Hence, the sequence ${\bf a}$ satisfies \eqref{sumt_n}.
The proof is completed.
\end{proof}

\subsection{Extreme boundary}

\

Consider the set $K$ of all sequences with the property \eqref{sumt_n} with the sum no larger  than~$1.$
Then $K$ can be seen as a subset of $\ell_1\equiv c_0^*.$
It is clear that $K$ is bounded, and closed with respect to weak topology on $\ell_1.$
So, it is a convex $(c_0, \ell_1)$-compact set. By the Krein-Milman theorem $K$ is a weak-closure of the
convex hull of extreme boundary, that is, $K=\overline{{\rm co}({\rm ext}K)}^{w}.$

In this subsection we describe the extreme boundary ${\rm ext}K$ of $K.$

Let $k_1, \ldots, k_n, \ldots$ be a sequence of integers such that $k_n\ge 2$ for all $n\ge 1.$
Set

\begin{align}\label{k-1}
{\bf a} & =\left(\underbrace{\frac{1}{k_1}, \ldots, \frac{1}{k_1}}_{k_1-1}, \underbrace{\frac{1}{k_1k_2}, \ldots, \frac{1}{k_1k_2}}_{k_2-1},\ldots, \underbrace{\frac{1}{k_1\cdots k_n}, \ldots, \frac{1}{k_1\cdots k_n}}_{k_n-1}, \ldots\right).
\end{align}

\begin{theorem}\label{extreme}
Let ${\bf a}\in K.$ Then ${\bf a}$ is an extreme point of $K$ if and only if it is in the form \eqref{k-1}.
\end{theorem}

Let us first check that a sequence  ${\bf a}$ defined as in \eqref{k-1} belong to  $K.$

\begin{proposition}\label{inK}
A sequence  ${\bf a}$ defined as in \eqref{k-1} is  in $K.$
\end{proposition}

\begin{proof} Let us first show that
\begin{align*}
\sum\limits_{n=1}^\infty\frac{k_n-1}{k_1\cdots k_n}=1.
\end{align*}
We will check that
\begin{align}\label{sn}
s_n & = \sum\limits_{i=1}^n\frac{k_i-1}{k_1\cdots k_i}=1-\frac{1}{k_1\cdots k_n}
\end{align}
for all $n\ge 1.$ Taking into account
\begin{align*}
\frac{k_i-1}{k_1\cdots k_i} & = \frac{1}{k_1\cdots k_{i-1}}-\frac{1}{k_1\cdots k_i},\, i\ge 2,
\end{align*}
we obtain that
\begin{align*}
s_n & = \left(1-\frac{1}{k_1}\right)+\left(\frac{1}{k_1}-\frac{1}{k_1 k_2}\right)+\ldots +\left(\frac{1}{k_1\cdots k_{n-1}}-\frac{1}{k_1\cdots k_n}\right)=
1-\frac{1}{k_1\cdots k_n},
\end{align*}
which gives us \eqref{sn}. Equality \eqref{sn} also shows that
\begin{align*}
\sum\limits_{m\ge 2}\frac{k_m-1}{k_1\cdots k_m}= \frac{1}{k_1\cdots k_{n-1}}
\end{align*}
for all $n\ge 2,$ which implies that ${\bf a}$ satisfies \eqref{sumt_n}.
The proof is completed.
\end{proof}

Let $E$ be a linear space and let $Q$ be its convex subset.
Recall that a convex subset $F$ of $Q$ is called
a face of $Q$ if
for any
 $x\in F$ and $y,z \in  Q,$ $0<\lambda<1$ such that $x=\lambda y+(1-\lambda)z,$
we have  $y,z \in F.$

For a fixed  integer $m\ge 2$ set
\begin{align*}
K_m & =\left\{{\bf a}=(a_n)_{n\ge 1}\in K: a_1=\ldots=a_{m-1}=\frac{1}{m}\right\}.
\end{align*}
Setting $k_n=m$ for all $n\ge 1$ in \eqref{k-1}, we see that ${\bf a}\in K_m.$ Furthermore, $K_m$ is a non-empty
convex compact subset in $K.$

\begin{proposition}\label{face}
$K_m$ is a face of $K.$
\end{proposition}

\begin{proof} Let ${\bf b}, {\bf c} \in K$ and $0< \lambda <1$ be such that $\lambda {\bf b} +(1-\lambda){\bf c}\in K_m.$

Let us first consider a case $m=2.$ By \eqref{sumt_n} we have that $b_1\le 1-b_1,$ that is, $b_1\le \frac{1}{2}.$
Likewise, $c_1\le \frac{1}{2}.$ Since $\lambda b_1+(1-\lambda)c_1=\frac{1}{2},$ it follows that $b_1=c_1=\frac{1}{2}.$ This means that
${\bf b}, {\bf c}\in K.$

Now suppose that $m\ge 3.$ We have that $\lambda b_i+(1-\lambda)c_i=\frac{1}{m}$ for all $i=1, \ldots, m-1.$
Note that $b_1\ge \ldots \ge b_{m-1}$ and $c_1\ge \ldots \ge c_{m-1}.$
Assume that $b_i>b_{i+1}$ for some $i\in \{1, \ldots, m-1\}.$
Then
\begin{align*}
\frac{1}{m} & = \lambda b_i+(1-\lambda)c_i> \lambda b_{i+1}+(1-\lambda)c_{i+1}=\frac{1}{m},
\end{align*}
which gives us a contradiction. Hence,  $b_1=\ldots=b_{m-1}$ and $c_1=\ldots =c_{m-1}.$
Using \eqref{sumt_n} we obtain  that
\begin{align*}
b_{m-1} & \le 1-\sum\limits_{i=1}^{m-1} b_i=1-(m-1)b_{m-1},
\end{align*}
that is,
$b_1=\ldots =b_{m-1}\le \frac{1}{m}.$ Likewise, $c_1=\ldots =c_{m-1}\le \frac{1}{m}.$ Taking into account
$\lambda b_1+(1-\lambda)c_1=\frac{1}{m},$
we obtain that
\[
b_1=c_1=\ldots =b_{m-1}=c_{m-1}=\frac{1}{m}.
\]
This means that ${\bf b}, {\bf c} \in K_m,$ that is, $K_m$ is a face. The proof is completed.
\end{proof}

Let $m\ge 2.$ Define a mapping $\Pi_m: K\to K_m$ as follows
\begin{align}\label{aff}
\Pi_m({\bf a})=\left(\underbrace{\frac{1}{m}, \ldots, \frac{1}{m}}_{m-1}, \frac{1}{m}a_1, \frac{1}{m}a_2, \ldots\right),\, {\bf a}\in K.
\end{align}

The next property directly follows from the definition.
The mapping $\Pi_m$ defined as in \eqref{aff} is an affine isomorphism from $K$ onto $K_m.$ The inverse  of the mapping defined as in \eqref{aff} is in the form
\begin{align}\label{aff-}
{\bf a}=\left(\frac{1}{m}, \ldots, \frac{1}{m}, a_m, \ldots\right)  \in K_m
\mapsto \Pi_m^{-1}({\bf a})=\left(ma_m, ma_{m+1}, \ldots\right)\in K.
\end{align}

\begin{proposition}\label{crit}
Let ${\bf a}\in K$ be such that $a_1=\ldots =a_{m-1}>a_m,$ where $m\ge 2.$ Then the following assertions are equivalent.
\begin{itemize}
\item[(a)] ${\bf a}\in {\rm ext}K;$
\item[(b)] $a_1=\ldots =a_{m-1}=\frac{1}{m}$ and $\left(ma_m, ma_{m+1}, \ldots\right)\in {\rm ext}K.$
\end{itemize}
\end{proposition}

\begin{proof} (a)$\Rightarrow$(b).  Let ${\bf a}\in {\rm ext}K$ and $a=a_1=\ldots =a_{m-1}>a_m.$ By \eqref{sumt_n} we have that
\[
a=a_{m-1}\le 1 - \sum\limits_{i=1}^{m-1}a_i=1 -(m-1)a,
\]
that is, $a\le \frac{1}{m}.$ Assume that  $a<\frac{1}{m}.$
We will find $0<\delta, \varepsilon<1$ such that
\[
{\bf b}=\left((1+\delta)a_1, \ldots, (1+\delta)a_{m-1}, (1-\varepsilon)a_m, (1-\varepsilon)a_{m+1},\ldots\right),
\]
\[
{\bf c}=\left((1-\delta)a_1, \ldots, (1-\delta)a_{m-1}, (1+\varepsilon)a_m, (1+\varepsilon)a_{m+1},\ldots\right),
\]
both should be  in $K.$ We must take $\varepsilon$  and $\delta$ such that
\[
(m-1)\delta a=\varepsilon(1-(m-1)a),
\]
in order to guarantee $\|{\bf b}\|_1=\|{\bf c}\|_1=1.$
Taking into account  $a<1-(m-1)a$ and $a=a_{m-1}>a_m$ we can choose these numbers such that
\begin{align*}
(1+\delta)a_{m-1} & \le (1-\varepsilon)(1-(m-1)a),\,\, (1-\delta)a_{m-1}\ge (1+\varepsilon)a_m.
\end{align*}
The last two conditions ensure that both ${\bf b}$ and ${\bf c}$
are non-increasing sequence and  satisfy~\eqref{sumt_n}.

Since ${\bf a}=\frac{1}{2}({\bf b}+{\bf c})$ and ${\bf b}\neq {\bf c},$ we obtain a contradiction with that
${\bf a}\in {\rm ext}K.$ So, $a_1=\ldots =a_{m-1}=\frac{1}{m}.$

Since affine isomorphism preserves extreme points of convex sets,  by Proposition~\ref{face}, we obtain that
$\left(ma_m, ma_{m+1}, \ldots\right)=\Pi_m^{-1}({\bf a})\in {\rm ext}K.$

(b)$\Rightarrow$(a). Let ${\bf a}\in K$ be such that
$a_1=\ldots =a_{m-1}=\frac{1}{m}$ and ${\bf x}=\left(ma_m, ma_{m+1}, \ldots\right)\in {\rm ext}K.$
Let ${\bf b}, {\bf c} \in K$ and $0< \lambda <1$ be such that ${\bf a}=\lambda {\bf b} +(1-\lambda){\bf c}.$
By a similar way as in the proof of Proposition~\ref{face} we obtain that
\begin{align*}
b_1=c_1=\ldots =b_{m-1}=c_{m-1}=\frac{1}{m}.
\end{align*}
This means that ${\bf b}, {\bf c}\in K_m.$
Taking into account that
\begin{align*}
{\bf x} & =\Pi_m^{-1}({\bf a})=\lambda\Pi_m^{-1}({\bf b})+(1-\lambda)\Pi_m^{-1}({\bf c}),
\end{align*}
${\bf x}=\left(ma_m, ma_{m+1}, \ldots\right)\in {\rm ext}K$
 and $\Pi_m$ is affine isomorphism, we obtain that
$\Pi_m^{-1}({\bf b})=\Pi_m^{-1}({\bf c}).$ Hence, ${\bf a}={\bf b}={\bf c},$ that is,
${\bf a}\in {\rm ext}K.$ The proof is completed.
\end{proof}

Now we are in position to present the proof of Theorem~\ref{extreme}.

\begin{proof}[Proof of Theorem \ref{extreme}]
'if' part. Let ${\bf a}\in K$ be in the form \eqref{k-1} and let ${\bf a}=\lambda {\bf b}+(1-\lambda){\bf c},$ where
${\bf b}, {\bf c}\in K,$ $0<\lambda<1.$ Since ${\bf a}\in K_{k_1}$ and $K_{k_1}$ is a face, it follows that
${\bf b}, {\bf c}\in K_{k_1}.$ Thus $b_1=c_1=\ldots=b_{k_1-1}=c_{k_1-1}=\frac{1}{k_1}.$
Then
$\Pi_{k_1}^{-1}({\bf a})$ is also in the form \eqref{k-1}, that is, $\Pi_{k_1}^{-1}({\bf a})\in K_{k_2},$ where $k_2\ge 2.$
Since $\Pi_{k_1}^{-1}({\bf a})=\lambda \Pi_{k_1}^{-1}({\bf b})+(1-\lambda)\Pi_{k_1}^{-1}({\bf c})\in K_{k_2}$ and $K_{k_2}$ is a face, we obtain that $\Pi_{k_1}^{-1}({\bf a})=\Pi_{k_1}^{-1}({\bf b})=\Pi_{k_1}^{-1}({\bf c}),$ in particular,
\begin{align*}
b_{k_1}=c_{k_1}=\ldots=b_{k_2-1}=c_{k_2-1}=\frac{1}{k_1k_2}.
\end{align*}
Continuing by this way we can find a sequence of integers $\{k_n\}_{n\ge 1}$ such that
$\Pi_{k_n}^{-1}\cdots \Pi_{k_1}^{-1}({\bf a})=\Pi_{k_n}^{-1}\cdots \Pi_{k_1}^{-1}({\bf b})=\Pi_{k_n}^{-1}\cdots \Pi_{k_1}^{-1}({\bf c})$ for all $n\ge 1.$ Thus
\begin{align*}
b_{k_n}=c_{k_n}=\ldots=b_{k_{n+1}-1}=c_{k_{n+1}-1}=\frac{1}{k_1\cdots k_{n+1}}
\end{align*}
for all $n\ge 1,$ that is, ${\bf b}={\bf c}={\bf a}.$ This means that  ${\bf a}\in {\rm ext}K.$

'only if' part. Let ${\bf a}\in {\rm ext}K.$ Let $k_1\ge 2$ be an integer  such that
$a_1=\ldots =a_{k_1-1}>a_{k_1}.$ By Proposition~\ref{crit} we obtain that
$a_1=\ldots =a_{m-1}=\frac{1}{k_1}$ and $\left(k_1a_{k_1}, k_1a_{k_1+1}, \ldots\right)\in {\rm ext}K.$
Now we can find an integer $k_2\ge 2$ such that
$k_1a_{k_1}=\ldots =k_1a_{k_2-1}>k_1a_{k_2}.$ Again using Proposition~\ref{crit} we obtain that
$k_1a_{k_1}=\ldots =k_1a_{k_1-1}=\frac{1}{k_2}$ and $\left(k_2k_1a_{k_2}, k_2k_1a_{k_2+1}, \ldots\right)\in {\rm ext}K.$
So,
\begin{align*}
a_1=\ldots =a_{k_1-1}=\frac{1}{k_1},\,\,\, a_{k_1}=\ldots =a_{k_2-1}=\frac{1}{k_1k_2}.
\end{align*}
Continuing by this way we obtain that ${\bf a}$ is in the form \eqref{k-1}.
The proof is completed.
\end{proof}

\begin{remark} Setting $k_n=2,$ $n\ge 1$  in \eqref{k-1} we obtain that
\begin{align*}
{\bf a} & =\left(\frac{1}{2}, \ldots, \frac{1}{2^n}\ldots\right)\in {\rm ext}K.
\end{align*}
 Then \eqref{dec-r} rewrites as follows
\begin{align*}
r=\sum\limits_{n=1}^\infty \frac{\varepsilon_n}{2^n}, \,\,\, \varepsilon_n \in \{0,1\}, n\ge 1,
\end{align*}
which gives the dyadic expansion of real number $r\in [0,1].$

Further, if we take $k_n=m,$ $n\ge 1,$ where $m \ge 2,$   using \eqref{k-1}, from  \eqref{dec-r} we obtain that
\begin{align*}
r=\sum\limits_{n=1}^\infty \frac{\epsilon_n}{m^n}, \,\,\, \epsilon_n \in \{0,1, \ldots, m-1\}, n\ge 1,
\end{align*}
which gives the $m$-ary expansion of real number $r\in [0,1].$

In general, if we take an arbitrary sequence $k_1, \ldots, k_n, \ldots$ with $k_n\ge 2,$ $n\ge 1,$
\eqref{dec-r} implies that
\begin{align*}
r=\sum\limits_{n=1}^\infty \frac{\epsilon_n}{k_1\cdots k_n}, \,\,\, \epsilon_n \in \{0,1, \ldots, k_n-1\}, n\ge 1.
\end{align*}
\end{remark}

\subsection*{Acknowledgment}

The second author was partially supported by the Russian Ministry of Education and Science, agreement no. 075-02-2025-1633.

\end{document}